\documentclass[12pt,a4paper,reqno, dvipsnames]{article}

\usepackage{graphicx} 

\usepackage{mathtools}
\usepackage{amsmath}
\usepackage{amssymb}
\usepackage{amsthm}
\usepackage{tikz,pgfplots}
\usetikzlibrary{decorations.markings}
\usepackage{comment}
\usepackage{url}
\usepackage{physics}
\usepackage{subfig}
\usepackage{float}

\usepackage{thmtools}
\usepackage{thm-restate}

\usepackage{graphicx,tikz}
\usetikzlibrary{shadows}

\usepackage{caption}
\usepackage{todonotes}
\usetikzlibrary{arrows}
\usepackage[utf8]{inputenc}
\usepackage[english]{babel}
\usepackage{array}
\usepackage{booktabs}
\usepackage{mdframed}
\usepackage{enumitem}
\usepackage{ragged2e}

\usepackage{soul}
\tikzset{Bullet/.style={fill=black,draw,color=#1,circle,minimum size=20pt,scale=.5}}

\usepackage{hyperref}

\usepackage{csquotes}

\hypersetup{
  colorlinks   = true, 
  urlcolor     = blue, 
  linkcolor    = black, 
  citecolor   = black 
}
\setstcolor{red}

\newtheorem{theorem}{Theorem}

\newtheorem{corollary}[theorem]{Corollary}
\newtheorem{lemma}[theorem]{Lemma}
\newtheorem{conjecture}[theorem]{Conjecture}

\newtheorem{proposition}[theorem]{Proposition}
\newtheorem{claim}{Claim}

\title{{Eigenvalue bounds for distance-edge colourings}}
\author{Aida Abiad\thanks{\texttt{a.abiad.monge@tue.nl}, Department of Mathematics and Computer Science, Eindhoven University of Technology, The Netherlands} 
\thanks{Department of Mathematics and Data Science, Vrije Universiteit Brussel, Belgium}
\and
Harper Reijnders\thanks{\texttt{l.e.r.m.reijnders@tue.nl}, Department of Mathematics and Computer Science, Eindhoven University of Technology, The Netherlands}}
\date{}

\begin{document}

\maketitle

\begin{abstract}
For a fixed positive integer $t$, we consider the graph colouring problem in which edges at distance at most $t$ are given distinct colours. We obtain sharp lower bounds for the distance-$t$ chromatic index, the least number of colours necessary for such a colouring. Our bounds are of algebraic nature; they depend on the eigenvalues of the line graph and on a polynomial which can be found using integer linear programming methods. We provide several graph classes that attain equality for our bounds, and also present some computational results which illustrate the bound's performance. Lastly, we investigate the implications the spectral approach has for the Erdős-Nešetřil conjecture, and derive some conditions which a graph must satisfy if we could use it to obtain a counter example through the proposed spectral methods.\\

\end{abstract}


\section{Introduction}

Given a graph $G = (V, E)$ and a positive integer $t$, a \emph{distance-$t$ edge-colouring} of $G$ is a colouring of the edges of $G$ such that no two edges within distance $t$ are given the same colour. Here, the \emph{distance between two edges} is defined as the length of the shortest path which contains a vertex from both edges. Incident edges have distance $1$ and thus a distance-$1$ edge-colouring is just a proper edge-colouring. The \emph{distance-$t$ chromatic index} of a graph $G$, denoted by $\chi_t'(G)$, is the smallest number of colours required to colour all the edges of $G$, such that no two edges within distance $t$ share a colour. The \emph{distance-$t$ chromatic number} of graph, denoted by $\chi_t(G)$, is the smallest number of colours required to colour all the vertices of $G$, such that no two vertices within distance $t$ share a colour. The \emph{$t$-th power of a graph} $G$, denoted by $G^t$, is a graph which has the same vertex set as $G$, and two vertices are adjacent if and only if they are within distance $t$ in $G$. The \emph{line graph} of a graph $G=(V,E)$, denoted by $L(G)$, has as vertex set $E$, with two vertices adjacent if and only if their corresponding edges in $V$ are adjacent. The distance-$t$ edge-colouring problem is related to colouring powers of graphs as follows
\begin{equation}\label{eq:chrom_index_number_relation}
 \chi_t'(G) = \chi_t(L(G)) = \chi_1(L(G)^t).
\end{equation}

Research into the classical $t=1$ case goes way back. In 1964, Vizing, in the form of his eponymous theorem, proved that $\chi'(G) \in \{\Delta, \Delta+1\}$, where $\Delta$ is the maximum degree of $G$. However, determining whether $\chi'(G)$ is $\Delta$ or $\Delta+1$ remains an NP-complete problem \cite{holyer_np-completeness_1981}, even for regular graphs \cite{leven_np_1983}.

Distance-$2$ edge-colourings (also known as \emph{strong edge-colourings}) have a rich history which goes back to 1985 (cf. \cite{faudree_induced_1989}), and its research has been largely focused on the following conjecture by Erdős and Nešetřil:

\begin{conjecture}{(\cite[Section 1]{faudree_induced_1989})}
\label{con:erdos_nesetril}
    \begin{equation*}
        \chi_2'(G) \leq \begin{cases}
            \frac{5\Delta^2}{4} &\text{ if } \Delta \text{ is even,} \\
            \frac{5\Delta^2-2\Delta+1}{4}  &\text{ if } \Delta \text{ is odd.}
        \end{cases}
    \end{equation*}
\end{conjecture}

While much effort has been devoted to this conjecture (see e.g. \cite{molloy_bound_1997,bruhn_stronger_2015,hurley_improved_2021}), it remains wide open. Additionally, quite some work has been done on determining $\chi_2'(G)$ for specific classes of graphs (see \cite{wang_strong_2018} for an overview) and on bounding $\chi_2'(G)$ for random graphs (see e.g. \cite{P1998,V2002,CN2004,FKS2005}).

The general distance-$t$ chromatic index was first introduced by Skupień in the early 1990's \cite{skupien_maximum_1995}, and since then it has been studied from numerous angles. Skupień looked into the distance-$t$ chromatic index of paths, cycles and trees \cite{skupien_maximum_1995}, as well as of hypercube graphs, using a connection to coding theory \cite{skupien_bch_2007}. The distance-$t$ chromatic index of hypercube graphs, as well as meshes, was later determined exactly by Drira et al. \cite{drira_distance_2013}, who were motivated by a connection to collision-free communication in wireless sensor networks. Ito et al. \cite{IKZN2007} developed two polynomial-time algorithms for finding distance-$t$ edge-colourings of partial $k$-trees and planar graphs. Recently, this parameter has garnered quite some attention in the form of a distance-$t$ Erdős-Nešetřil problem, the distance-$t$ analog of Conjecture \ref{con:erdos_nesetril}. This entails obtaining upper bounds of the form $\chi'_t(G) \leq b\Delta^t$. A result by Kang and Manggala \cite{kang_distance_2012} provided a construction which shows $\chi'_t(G) = \Omega(\Delta^t)$, motivating research into bounds of the form $b\Delta^t$. Bounds of this form have since then been obtained using the probabilistic method \cite{kaiser_distance-t_2014, cambie_maximizing_2022}.  Additionally, the asymptotic behaviour of $\chi'_t(G)$ in terms of $\Delta$ has been investigated for graphs of given girth \cite{mahdian_strong_2000, kaiser_distance-t_2014} and for graphs without a specific cycle length \cite{kang_distance_2017}, using (probabilistic) constructions.

Until now most research has been devoted to deriving upper bounds for the distance-$t$ chromatic number. Largely done using probabilistic methods. In this paper, we investigate the distance-$t$ chromatic index of graphs using the spectral graph theory machinery. This allows us to derive two sharp lower bounds on the distance chromatic index that depend only on the eigenvalues of the line graph (and not the power of the line graph, the spectrum of which is in general unrelated to the spectrum of the line graph itself), and on the choice of a real polynomial of degree $t$. Regarding the latter, we show that finding the best possible lower bound for a given graph reduces to solving a linear integer optimization problem. For small $t$, we prove that our bounds are sharp for some graph classes. For general $t$, we illustrate the tightness of our bounds by means of computational results. Lastly, we investigate what the spectral approach says about Conjecture \ref{con:erdos_nesetril} and its distance-$t$ analog.

The remainder of this paper is organized as follows. In Section \ref{sec:preliminaries} we fix the notation and recall some preliminary results on line graphs. In Section \ref{sec:spec_chiprimet} we present 
two new lower bounds on $\chi_t'(G)$, which depend on some polynomial optimization, and we derive the best polynomial for one of the bounds for $t=2,3$. In Section \ref{sec:chiprime_performance} we look at the performance of the new eigenvalue bounds, both theoretically and computationally. In particular, we prove that the second lower bound is tight for several graph classes. Finally, in Section \ref{sec:erdos_nesetril} we investigate the distance-$t$ Erdős-Nešetřil problem using our spectral approach, and derive strict conditions on the spectrum that any counterexample found through this spectral method would need to satisfy.

\section{Preliminaries}\label{sec:preliminaries}
Throughout, let $G = (V,E)$ be a simple, undirected graph on $\abs{V} = n$ with $\abs{E} = m$ edges. We denote by $A$ the adjacency matrix of $G$, and we call the eigenvalues of $A$ the \emph{adjacency spectrum} of $G$. We denote the adjacency eigenvalues by $\lambda_1 \geq \cdots \geq \lambda_n$, and when we only want to consider distinct eigenvalues, we instead use the notation $\theta_0 > \theta_1 > \cdots > \theta_d$, where $d+1$ is the amount of distinct eigenvalues. When talking about adjacency eigenvalues of the line graph $L(G)$, we will instead use $\lambda'_1 \geq \cdots \geq \lambda'_m$ and $\theta'_0 > \cdots > \theta'_{d'}$.

The \emph{degree} of a vertex $v$, denoted $d(v)$, is the number of edges that $v$ is incident to. The \emph{maximum degree} of $G$, denoted by $\Delta$, is the maximum of $d(v)$ over all vertices $v \in V$. If $d(v) = k$ for all $v \in V$, then we call the graph \emph{$k$-regular}.

Line graphs are a frequently used tool which allow us to connect edge and vertex versions of the same type of graph parameter, see e.g. \cite{skupien_maximum_1995, skupien_bch_2007, kang_distance_2012}. There are some preliminary results that we need in order to exploit this link. 


    \begin{lemma}{(\cite[Chapter 3]{biggs_algebraic_1974})}\label{lem:LG_reg}
        Let $G$ be a $k$-regular graph, then $L(G)$ is $(2k-2)$-regular.
    \end{lemma}

    Since the new bounds will use the spectrum of the line graph, it is useful to characterize it in terms of the spectrum of the original graph. While in general no such connection is known, for regular graphs there is the following characterization.

    \begin{corollary}{(\cite[Theorem 3.8]{biggs_algebraic_1974})}\label{cor:LG_spec}
        Let $G$ be a $k$-regular graph with $n$ vertices and $m = \frac{nk}{2}$ edges and adjacency eigenvalues $\lambda_1 \geq \cdots \geq \lambda_n$. Then the line graph $L(G)$ has adjacency eigenvalues $\lambda_1+k-2 \geq \cdots \geq \lambda_n+k-2$ and the eigenvalues $-2$ an additional $m-n$ times.
    \end{corollary}

    In a similar vein, there is in general no connection between the spectrum of a (line) graph and the power graph of said (line) graph. See e.g. \cite{dg13} and \cite[Section 2]{ACFNZ2022}. This gives a strong argument for our bounds in Section \ref{sec:spec_chiprimet}, as they depend only on the spectrum of the original (line) graph and do not require the power graph.

    We finally note that the spectrum of the line graph is equivalent to the spectrum of the signless Laplacian \cite[Proposition 1.4.1]{brouwer_spectra_2012}, but we do not use this connection in this paper.

\section{Spectral lower bounds}\label{sec:spec_chiprimet}

    This section presents two lower bounds on the distance-$t$ chromatic index that depend on the adjacency eigenvalues of $L(G)$ (which, for regular graphs, are directly related to the adjacency eigenvalues of $G$, by Corollary \ref{cor:LG_spec}) and on the choice of a polynomial of degree $t$. To derive them we leverage some known spectral bounds for distance colourings of graphs by the first author and collaborators \cite{ACFNZ2022}. For $t=2,3$ we also show the best choice of polynomial that optimizes our bounds and obtain closed bound expressions. The bounds' tightness will be showed in Section \ref{sec:chiprime_performance}.

    By $\mathbb{R}_t[x]$ we denote the set of polynomials of degree at most $t$ in the variable $x$.
    \begin{theorem}[First inertial-type bound]\label{thm:1st_inertial_chiprime}
            Let $G =(V,E)$ be a graph, let $L(G)$ be its line graph with adjacency matrix $A'$ and eigenvalues $\lambda'_1 \geq \lambda'_2 \geq \cdots \geq \lambda'_m$. Let $p \in \mathbb{R}_t[x]$ with corresponding parameters $W'(p) := \max_{e \in E} \{(p(A'))_{ee}\}$ and $w'(p) := \min_{e \in E} \{(p(A'))_{ee}\}$. Then,
            \begin{equation*}
                \chi'_t(G) \geq \frac{\abs{E}}{\min\{ \abs{\{i : p(\lambda'_i) \geq w'(p) \}},\abs{\{i : p(\lambda'_i) \leq W'(p) \}}\}}.
            \end{equation*}
    \end{theorem}
    \begin{proof}
        We apply \cite[Equation (18)]{ACFNZ2022} to $L(G)$. This, plus \eqref{eq:chrom_index_number_relation}, gives the desired bound.
    \end{proof}

    We call a graph $G = (V,E)$ \emph{$t$-partially walk-regular} if for each positive integer $\ell \leq t$ and any two vertices $u,v \in V$, the amount of paths of length $\ell$ from $u$ to $v$ is independent of choice of vertices $u,v$.

    \begin{theorem}[Second inertial-type bound]\label{thm:2nd_inertial_chiprime}
        Let $G$ be a graph such that $L(G)$ is $t$-partially walk-regular. Let $L(G)$ have adjacency eigenvalues $\lambda'_1 \geq \cdots \geq \lambda'_m$. Let $p \in \mathbb{R}_t[x]$ such that $\sum_{i=1}^m p(\lambda'_i) = 0$. Then,
        \begin{equation*}
            \chi'_t(G) \geq 1 + \max\left(\frac{\abs{\{j:p(\lambda'_j) <0\}}}{\abs{\{j:p(\lambda'_j) >0\}}}\right).
        \end{equation*}
    \end{theorem}
    \begin{proof}
        By assumption, $L(G)$ is $t$-partially walk-regular, thus we can apply \cite[Theorem 4.2]{ACFNZ2022}. The bound then follows immediately from \eqref{eq:chrom_index_number_relation}.    
    \end{proof}
    Note that, unlike the other bounds in this section where we are able to use Lemma \ref{lem:LG_reg} to assert regularity of $L(G)$, we cannot know if $L(G)$ is $t$-partially walk-regular by simply checking if $G$ is $t$-partially walk-regular. Hence, we need the condition that $L(G)$ is $t$-partially walk-regular for Theorem \ref{thm:2nd_inertial_chiprime} to hold.

    We next present two Hoffman-type bounds Theorem \ref{thm:ratio_chiprime} holds for general graphs, whereas Theorem \ref{thm:ratio_chiprime_reg} holds only for regular graphs, but gives a stronger bound due to a rounding argument. 
    \begin{theorem}[General Hoffman-type bound]\label{thm:ratio_chiprime}
        Let $G$ be a graph, let $L(G)$ be its line graph with adjacency matrix $A'$ and eigenvalues $\lambda'_1 \geq \lambda'_2 \geq \cdots \geq \lambda'_m$. Let $p \in \mathbb{R}_t[x]$ with corresponding parameters $W'(p) := \max_{e \in E}$ $\{(p(A'))_{ee}\}$ and $ \lambda'(p):=\min_{2\leq i \leq m} \{p(\lambda'_i)\}$, and assume $p(\lambda'_1) > \lambda'(p)$. Then, 
        \begin{equation}\label{eq:ratio_chiprime}
            \chi'_t(G) \geq \frac{p(\lambda'_1) - \lambda'(p)}{W'(p) - \lambda'(p)}.
        \end{equation}
    \end{theorem}
    \begin{proof}
        By assumption, the conditions are met to apply \cite[Theorem 4.3]{ACFNZ2022} to $L(G)$. This, plus \eqref{eq:chrom_index_number_relation}, gives the desired bound.
    \end{proof}

    \begin{theorem}[Hoffman-type bound]\label{thm:ratio_chiprime_reg}
        Let $G$ be a regular graph, let $L(G)$ be its line graph with adjacency matrix $A'$ and eigenvalues $\lambda'_1 \geq \lambda'_2 \geq \cdots \geq \lambda'_m$. Let $p \in \mathbb{R}_t[x]$ with corresponding parameters $W'(p) := \max_{e \in E}$ $\{(p(A'))_{ee}\}$ and $ \lambda'(p):= \min_{2\leq i \leq m} \{p(\lambda'_i)\}$, and assume $p(\lambda'_1) > \lambda'(p)$. Then, 
        \begin{equation*}
            \chi'_t(G) \geq \frac{\abs{E}}{\left\lfloor\abs{E}\frac{W'(p) - \lambda'(p)}{p(\lambda'_1) - \lambda'(p)}\right\rfloor}.
        \end{equation*}
    \end{theorem}    
    \begin{proof}
        Lemma \ref{lem:LG_reg} tells us that $L(G)$ is $2(k-1)$-regular, meaning \cite[Theorem 2]{ANR2024} is applicable to $L(G)$. This, plus \eqref{eq:chrom_index_number_relation}, gives us the desired bound.
    \end{proof}

    Since the bounds in this section were originally derived for $\chi_t(G)$ and $\alpha_t(G)$ (where $\alpha_t(G) = \alpha(G^t)$, the independence number of the power graph) using interlacing techniques (see \cite{abiad_k_2019, ACFNZ2022}), one might expect to be able to directly derive bounds for $\chi_t'(G)$ using similar techniques. However, these techniques make heavy use of the way the vertices are described in the adjacency matrix, and this does not work for edges, which the parameter $\chi_t'(G)$ is based on.

\subsection{Optimization of the bounds on \texorpdfstring{$\chi'_t(G)$}{chit'(G)}}\label{sec:chiprime_MILP}
    
    As the above bounds on $\chi_t'(G)$ from this section were derived by using bounds on $\chi_t(L(G))$, we too can use mixed integer linear programming (MILP) and linear programming (LP) implementations for these bounds by using the existing implementations for the corresponding analogous eigenvalue bounds on $\chi_t(G)$. Indeed, simply use as input the graph $L(G)$ instead of $G$ in the following mixed linear programs:
    \begin{itemize}
       \item \textbf{First inertial type bound (Theorem \ref{thm:1st_inertial_chiprime})} MILP (20) from \cite{ACFNZ2022}.
       \item \textbf{Second inertial type bound (Theorem \ref{thm:2nd_inertial_chiprime}):} MILP (3.18) from \cite{zeijlemaker_optimization_2024}. (Originally appeared incorrectly in \cite{ACFNZ2022}).
       \item \textbf{Hoffman-type bound (Theorems \ref{thm:ratio_chiprime} and \ref{thm:ratio_chiprime_reg})}: LP (4) and LP (5) from \cite{ANR2024}. There is one technicality that needs to be considered regarding LP (5) from \cite{ANR2024}: this LP requires $G$ to be $t$-partially walk-regular, which in our case means we require $L(G)$ to be $t$-partially walk-regular. We note that \cite[LP (5)]{ANR2024} is based on an LP formulated by Fiol \cite{fiol_new_2020}.
    \end{itemize}

    \subsection{Closed formulas for the bounds on \texorpdfstring{$\chi'_2(G)$}{chi'2(G)} and \texorpdfstring{$\chi'_3(G)$}{chi'3(G)} }\label{sec:closedformulas}
    In addition to the above bounds for general polynomials $p$ and general $t$, we can also obtain specific bounds for $\chi'_2(G)$ and $\chi'_3(G)$ as follows.
    \begin{corollary}\label{cor:ratio_chiprime_2}
        Let $G$ be a graph with $n$ vertices. Denote by $\theta'_0 > \cdots > \theta'_{d'}$ the distinct adjacency eigenvalues of $L(G)$, with $d' \geq 2$ and by $\Delta'$ the maximum degree of $L(G)$. Let $\theta'_i$ be the largest eigenvalue such that $\theta'_i \leq -\frac{\Delta'}{\theta'_0}$.
Then, 
        \begin{equation}
            \chi'_2(G) \geq \frac{(\theta'_0 - \theta'_i)(\theta'_0 - \theta'_{i-1})}{\Delta' + \theta'_i\theta'_{i-1}}.
        \end{equation}
        No better bound on $\chi_2'(G)$ can be obtained by using Theorem \ref{thm:ratio_chiprime}.
    \end{corollary}
    \begin{proof}
        The bound follows directly from \cite[Theorem 3]{ANR2024} after applying \eqref{eq:chrom_index_number_relation}. What remains to be proven is that this is the best possible choice for Theorem \ref{thm:ratio_chiprime}. Suppose there was some polynomial $p \in \mathbb{R}_2[x]$ which, when used to obtain a bound from Theorem  \ref{thm:ratio_chiprime}, gives a stronger bound on $\chi'_2(G)$. Then this $p$ will also result in a stronger bound on $\chi_2(L(G))$ than the one from \cite[Theorem 3]{ANR2024}. This contradicts the optimality assertion from \cite[Theorem 3]{ANR2024}, thus no such $p$ can exist, and thus no better bound can be obtained from Theorem \ref{thm:ratio_chiprime}.
    \end{proof}

    \begin{corollary}\label{cor:ratio_chiprime_2_reg}
        Let $G$ be a $k$-regular graph. Let $2(k-1) = \theta'_0 > \cdots  > \theta'_{d'}$ be the distinct adjacency eigenvalues of $L(G)$, with $d' \geq 2$. Let $\theta'_i$ be the largest eigenvalue such that $\theta'_i \leq -1$. Then, \begin{equation}\label{eq:ratio_chiprime_2_reg}
                \chi_2'(G) \geq \frac{\abs{E}}{\left\lfloor\abs{E}\frac{\theta'_0 + \theta'_i\theta'_{i-1}}{(\theta'_0 - \theta'_i)(\theta'_0 - \theta'_{i-1})}\right\rfloor}.
            \end{equation}
        No better bound on $\chi_2'(G)$ can be obtained by using Theorem \ref{thm:ratio_chiprime_reg}.    
    \end{corollary}
    
    \begin{proof}
        Lemma \ref{lem:LG_reg} tells us that $L(G)$ is $2(k-1)$-regular, meaning \cite[Corollary 4]{ANR2024} is applicable to $L(G)$. Finally, we use \eqref{eq:chrom_index_number_relation} to obtain the desired bound. The proof of optimality is analogous to that of Corollary \ref{cor:ratio_chiprime_2} but now using the optimality assertion from \cite[Corollary 4]{ANR2024}, and is thus omitted.
    \end{proof}
    
    \begin{corollary}\label{cor:ratio_chiprime_3}
        Let $G$ be a $k$-regular graph. Let $A'$ be the adjacency matrix of $L(G)$ and let $2(k-1) = \theta'_0 >  \cdots > \theta'_{d'} = -2$ be the distinct adjacency eigenvalues of $L(G)$, with $d' \geq 3$. Let $\theta'_s$ be the largest eigenvalue such that $\theta'_s \leq - \frac{\theta'^2_0 + \theta'_0\theta'_{d'}-\Delta'_3}{\theta'_0 (\theta'_{d'} + 1)}$, where $\Delta'_3 := \max_{e\in E} \{((A')^3)_{ee}\}$. Then,
            \begin{equation*}
                \chi_3'(G) \geq \frac{\abs{E}}{\left\lfloor\abs{E}\frac{\Delta'_3 - \theta'_0(\theta'_s + \theta'_{s-1} + \theta'_{d'}) - \theta'_s\theta'_{s-1} \theta'_{d'}}{(\theta'_0 - \theta'_s)(\theta'_0 - \theta'_{s-1})(\theta'_0 - \theta'_{d'})}\right\rfloor}.
            \end{equation*}
        No better bound on $\chi_3'(G)$ can be obtained by using Theorem \ref{thm:ratio_chiprime_reg}.
    \end{corollary}

    \begin{proof}
       While we do not explicitly use it, we note that the proof of this corollary heavily relies on \cite[Theorem 11]{kavi_optimal_2023}.  
       Lemma \ref{lem:LG_reg} tells us that $L(G)$ is $2(k-1)$-regular, meaning \cite[Corollary 5]{ANR2024} is applicable to $L(G)$. Finally, we apply \eqref{eq:chrom_index_number_relation} to obtain the desired bound. The proof of optimality is analogous to that of Corollary \ref{cor:ratio_chiprime_2} but now using the optimality assertion of \cite[Corollary 5]{ANR2024}, and is thus omitted.
    \end{proof}

\raggedright
\section{Tightness and performance of the lower bounds} \label{sec:chiprime_performance}
\justifying
In this section we study how the bounds from Section \ref{sec:spec_chiprimet} perform. First, we present several graph families for which the Hoffman-type bounds on $\chi_t'(G)$ are sharp; these are infinite families of balanced bipartite products of cycles. Then we investigate how the bounds perform computationally, comparing them to each other, as well as to the true parameter value (see also Appendix).

For general $t$, it is hard to find the best possible Hoffman-type bound. Thus, in what follows, we focus on showing tightness for the Hoffman-type bound when $t=2,3$, and we do so by using the closed expressions that we derived in Section \ref{sec:spec_chiprimet}.

\subsection{Balanced bipartite product}

Let us first introduce the graph construction which we  use in this section. Let $G_1 = (V_1 \cup U_1, E_1), G_2 = (V_2 \cup U_2, E_2)$ be two \emph{balanced bipartite graphs}, that is, $\abs{U_i} = \abs{V_i}$ and $E_i \subset U_i \times V_i$ for $i=1,2$. Now consider a fixed vertex ordering on these graphs: $V_i = (v_1^i,\hdots,v_{n_i}^i), U_i = (u_1^i,\hdots, u_{n_i}^i)$. In our case we will only consider matching orderings (i.e.\ $(v_j^i,u_j^i) \in E_i$ for all $j$). The \emph{balanced bipartite product} of $G_1$ and $G_2$, denoted by $G_1 \bowtie G_2$ was introduced in \cite{kang_distance_2017}. It is defined as follows. Let $H = G_1 \bowtie G_2 = (V,E)$, then
\begin{align*}
    V := & \ V_1 \times V_2 \cup U_1 \times U_2, \\
    E := & \ \{(v_j^1,v^2)(u_j^1,u^2): j \in \{1,\hdots,n_1\},(v^2,u^2) \in E_2\} \cup\\
    & \ \{(v^1,v_j^2)(u^1,u_j^2): j \in \{1,\hdots,n_2\},(v^1,u^1) \in E_1\}.
\end{align*}
See Figure \ref{fig:balbiprod_example} for a small example of $C_4 \bowtie C_4$.
\begin{figure}
        \centering
        \subfloat{\begin{tikzpicture}
            \node[fill,circle, label=\textcolor{red}{$v_1^1$}] (v1) at (0,2){};
            \node[fill,circle, label=below:{\textcolor{blue}{$v_2^1$}}] (v2) at (2,0){};
            \node[fill,circle, label=below:{\textcolor{Dandelion}{$u_2^1$}}] (u2) at (0,0){};
            \node[fill,circle, label=\textcolor{Green}{$u_1^1$}] (u1) at (2,2){};
            \node[fill,circle, color=white] (X) at (0,-1.5){};
            
            \draw (u1)--(v1)--(u2)--(v2)--(u1);
        \end{tikzpicture}} \put(-5,73){$\bowtie$}
        \subfloat{\begin{tikzpicture}
            \node[fill,circle, label=\textcolor{red}{$v_1^2$}] (v1) at (0,2){};
            \node[fill,circle, label=below:{\textcolor{blue}{$v_2^2$}}] (v2) at (2,0){};
            \node[fill,circle, label=below:{\textcolor{Dandelion}{$u_2^2$}}] (u2) at (0,0){};
            \node[fill,circle, label=\textcolor{Green}{$u_1^2$}] (u1) at (2,2){};
            \node[fill,circle, color=white] (X) at (0,-1.5){};
            \draw(u1)--(v1)--(u2)--(v2)--(u1);
        \end{tikzpicture}} \put(0,73){$=$}
        \subfloat{\begin{tikzpicture}
            \node[fill,circle, label=below:{$(\textcolor{Dandelion}{u_2^1},\textcolor{Dandelion}{u_2^2})$}] (u1) at (0,0){};
            \node[fill,circle, label=below:{$(\textcolor{blue}{v_2^1},\textcolor{blue}{v_2^2})$}] (u2) at (3,0){};
            \node[fill,circle, label={$(\textcolor{Dandelion}{u_2^1},\textcolor{Green}{u_1^2})$}] (u3) at (3,3){};
            \node[fill,circle, label={$(\textcolor{blue}{v_2^1},\textcolor{red}{v_1^2})$}] (u4) at (0,3){};
            \node[fill,circle, label=below:{$(\textcolor{red}{v_1^1},\textcolor{blue}{v_2^2})$}] (u5) at (1.5,1.5){};
            \node[fill,circle, label=below:{$(\textcolor{Green}{u_1^1},\textcolor{Dandelion}{u_2^2})$}] (u6) at (4.5,1.5){};
            \node[fill,circle, label={$(\textcolor{red}{v_1^1},\textcolor{red}{v_1^2})$}] (u7) at (4.5,4.5){};
            \node[fill,circle, label={$(\textcolor{Green}{u_1^1},\textcolor{Green}{u_1^2})$}] (u8) at (1.5,4.5){};
            \draw (u1)--(u2)--(u6)--(u5)--(u1);
            \draw (u3)--(u4)--(u8)--(u7)--(u3);
            \draw (u1)--(u4);
            \draw (u2)--(u3);
            \draw (u5)--(u8);
            \draw (u6)--(u7);
        \end{tikzpicture}}
        \caption{The balanced bipartite product $C_4 \bowtie C_4 = Q_3$.}
        \label{fig:balbiprod_example}
    \end{figure}
Balanced bipartite products are very well structured, see \cite{kang_distance_2017} for more details, where their structure was used to determine the asymptotic supremum of $\chi_t(G)$ and $\chi_t'(G)$ over all $C_\ell$-free graphs. Of particular interest to us, is what their adjacency matrix looks like. The following result tells us exactly this.

\begin{proposition}\label{prop:balbiprod_adjacency}
    Let $G_1,G_2$ be two balanced bipartite graphs with $2r,2r'$ vertices respectively, and adjacency matrices $$\begin{pmatrix}
        O & A \\
        A^\top & O
    \end{pmatrix}, \begin{pmatrix}
        O & B \\
        B^\top & O
    \end{pmatrix}$$ respectively. Let $I$ be the identity matrix of size $2r$. Then $H= G_1 \bowtie G_2$ has adjacency matrix
    \begin{equation*}
        \tilde{A} = \begin{pmatrix*}
            O & A & O & B_{12} I &  \cdots & O & B_{1r'}I \\
            A^\top & O &  B_{21}I & O & \cdots & B_{r'1} & O \\
            O & B_{21} I & O & A & \cdots & O & B_{2r'} \\
            B_{12}I & O & A^\top & O  & \cdots & B_{r'2} & O \\
            \vdots & \vdots & \vdots & \vdots & \ddots & \vdots & \vdots \\
            O & B_{r'1} I & O & B_{r'2} I & \cdots & O & A \\
            B_{1r'} I & O & B_{2r'} I & O & \cdots & A^\top & O
        \end{pmatrix*}.
    \end{equation*}
\end{proposition}
\begin{proof}
    Let $G_1 = (V_1 \cup U_1, E_1)$ and $G_2 = (V_2 \cup U_2, E_2)$. First, we will order the vertices of $H$ such that the labeling agrees with $\tilde{A}$. Consider the subsets $\tilde{V}_j = \{(v_i^1,v_j^2): i \in [2r]\}, \tilde{U}_j = \{(u_i^1,u_j^2): i \in [2r]\}$. Inside these sets we use the induced ordering from the $v_i^1$ or $u_i^1$, and the sets themselves we order as $\tilde{V}_1,\tilde{U}_1,\tilde{V}_2,\hdots,\tilde{U}_{2r'}$. We can now see that these sets will correspond to the blocks in $\tilde{A}$. Since $G_1$ and $G_2$ are in matching order, $\tilde{V}_i \cup \tilde{U}_i$ induces the subgraph $G_1$, thus the diagonal of $\tilde{A}$ is indeed $\begin{pmatrix}
        O & A \\
        A^\top & O
    \end{pmatrix}$. Additionally, for any $i \in [2r], j \in [2r']$, $\tilde{V}_i \cup \tilde{V}_j$ induces the empty graph, the same holds for $\tilde{U}_i, \tilde{U}_j$. This corresponds to the $O$ blocks in $\tilde{A}$.

    Now we just need to consider the adjacency in the blocks corresponding to $\tilde{V}_i \cup \tilde{U}_j$ for $i \neq j$. Fix such an $i$ and $j$ and let us consider a pair of vertices $(v_{\ell_1}^1,v_i^2) \in \tilde{V}_i$ and $(u_{\ell_2}^1,u_j^2) \in \tilde{U}_j$. Since we assumed $i \neq j$, we know by definition that these vertices are adjacent if and only if both $\ell_1=\ell_2$ and $(v_i^2,u_j^2) \in E_2$. Hence, all off-diagonal entries in this block are $0$, and furthermore, the diagonal is constant, and is $1$ if $B_{ij}=1$, and $0$ otherwise. In other words, this block is of the form $B_{ij}I$.
\end{proof}

While we will not be using it here, it is worth noting that it follows directly from Proposition \ref{prop:balbiprod_adjacency} that the adjacency eigenvalues of both $G_1$ and $G_2$ interlace the adjacency eigenvalues of $G_1 \bowtie G_2$ .

In the rest of this section, we will focus exclusively on graphs of the form $H = C_q \bowtie C_{q'}$. We will show that our Hoffman-type bounds for $\chi'_2(G)$ and $\chi'_3(G)$ from Section \ref{sec:spec_chiprimet} are tight for an infinite family of balanced bipartite products of cycles. In particular, in Section \ref{sec:chi2_tight}, we show that there is an infinite family of graphs $C_q \bowtie C_{q'}$ for which Corollary \ref{cor:ratio_chiprime_2_reg} is tight. In Section \ref{sec:chi3_tight}, we look at a special class of balanced bipartite products known as \emph{Guo-Mohar graphs}, $GM(k) = C_4 \bowtie C_{2k}$, and show that Corollary \ref{cor:ratio_chiprime_3} is tight for $GM(2k)$.

\subsubsection{An infinite family of graphs for which Corollary \ref{cor:ratio_chiprime_2_reg} is tight} \label{sec:chi2_tight}

The next result shows that Corollary \ref{cor:ratio_chiprime_2_reg} is tight for an infinite family of graphs $C_q \bowtie C_{q'}$, with $q,q' = 0 \text{ (mod }4)$. 

\begin{proposition}\label{prop:balbiprod_tight}
    Let $H = C_{q} \bowtie C_{q'}$ be a balanced bipartite product, with $q,q' = 0 \text{ (mod }4)$. Then the Hoffman-type bound for $\chi'_2(H)$ (Corollary \ref{cor:ratio_chiprime_2_reg}) is tight if either:
    \begin{enumerate}
        \item $-2$ is not an eigenvalue of $H$,
        \item $-2$ is an eigenvalue of $H$ and $qq' \neq 0 \text{ (mod }5)$.
    \end{enumerate}
\end{proposition}

In order to prove Proposition \ref{prop:balbiprod_tight}, we need several preliminary results. First, let us show that a distance-$2$ edge colouring of $C_q \bowtie C_{q'}$ with $6$ colours exists. Henceforth, we will consider the vertices of $C_q$ to be ordered such that $(v_{i+1},u_i)$ is an edge for $i \in [1,\frac{q}{2}-1]$.

\begin{lemma}\label{lem:balbiprod_colour}
    Let $H = C_q \bowtie C_{q'}$ be a balanced bipartite product, with $q,q' = 0 \text{ (mod }4)$. Then $\chi'_2(H) \leq 6$.
\end{lemma}
\begin{proof}
    We colour the edges using $6$ colours as follows:
    \begin{align*}
        c((v_i^1,v_j^2),(u_i^1,u_j^2)) &= 3i + 3j \text{ (mod }6)\\
        c((v_{i+1}^1,v_j^2),(u_i^1,u_j^2)) &= 1+3i \text{ (mod }6)\\
        c((v_i^1,v_{j+1}^2),(u_i^1,u_j^2)) &= 2+3j \text{ (mod }6)
    \end{align*}
    where we use abuse of notation $v^1_{\frac{q}{2}+1} = v_1^1$ and $v^2_{\frac{q'}{2}+1} = v_1^2$. See Figure \ref{fig:6_colour_q4} for an example.
\end{proof}

\begin{figure}[H]
    \centering
    \begin{tikzpicture}
        \node[fill,circle] (u1) at (-1,0){};
        \node[fill,circle] (u2) at (0,0){};
        \node[fill,circle] (u3) at (-1,-1){};
        \node[fill,circle] (u4) at (0,-1){};
        \node[fill,circle] (u5) at (-1,-2){};
        \node[fill,circle] (u6) at (0,-2){};
        \node[fill,circle] (u7) at (-1,-3){};
        \node[fill,circle] (u8) at (0,-3){};
        \node[fill,circle] (v1) at (3,0){};
        \node[fill,circle] (v2) at (4,0){};
        \node[fill,circle] (v3) at (3,-1){};
        \node[fill,circle] (v4) at (4,-1){};
        \node[fill,circle] (v5) at (3,-2){};
        \node[fill,circle] (v6) at (4,-2){};
        \node[fill,circle] (v7) at (3,-3){};
        \node[fill,circle] (v8) at (4,-3){};
        \node[fill] (W1) at (-4,-1.5){};
        \node[fill] (W2) at (7,-1.5){};
        \draw[thick, color=red] (u1)--(u2);
        \draw[thick, color=blue] (u2)--(u3);
        \draw[thick, color=green] (u3)--(u4);
        \draw[thick, color=orange] (u4)--(u5);
        \draw[thick, color=red] (u5)--(u6);
        \draw[thick, color=blue] (u6)--(u7);
        \draw[thick, color=green] (u7)--(u8);
        \draw[thick, color=orange] (u8) to[bend left=65] (u1);
        \draw[thick, color=green] (v1)--(v2);
        \draw[thick, color=blue] (v2)--(v3);
        \draw[thick, color=red] (v3)--(v4);
        \draw[thick, color=orange] (v4)--(v5);
        \draw[thick, color=green] (v5)--(v6);
        \draw[thick, color=blue] (v6)--(v7);
        \draw[thick, color=red] (v7)--(v8);
        \draw[thick, color=orange] (v8) to[bend left=65] (v1);
        \draw[thick, color=cyan] (u2)--(v1);
        \draw[thick, color=cyan] (u4)--(v3);
        \draw[thick, color=cyan] (u6)--(v5);
        \draw[thick, color=cyan] (u8)--(v7);
        \draw[thick, color=magenta] (W1)--(u1);
        \draw[thick, color=magenta] (W1)--(u3);
        \draw[thick, color=magenta] (W1)--(u5);
        \draw[thick, color=magenta] (W1)--(u7);
        \draw[thick, color=magenta] (W2)--(v2);
        \draw[thick, color=magenta] (W2)--(v4);
        \draw[thick, color=magenta] (W2)--(v6);
        \draw[thick, color=magenta] (W2)--(v8);
    \end{tikzpicture}
    \caption{An example of how to $6$-colour part of $C_q \bowtie C_{q'}$ when $q,q'$ are multiples of $4$. The square nodes represent the remainder of the graph}
    \label{fig:6_colour_q4}
\end{figure}

Now we will set out to show that our spectral lower bound (Corollary \ref{cor:ratio_chiprime_2_reg}) gives a bound of $\chi'_2(C_q \bowtie C_{q'}) \geq 6$, and is thus tight. For this, we will make use of a slight extension of \cite[Corollary 29]{ANR2024} on the behavior of Corollary \ref{cor:ratio_chiprime_2_reg}, where we now take into account intermediate rounding.
\begin{lemma}\label{lem:chi2_ratio_behaviour}
    Let $G$ be a $k$-regular graph with distinct eigenvalues $k=\theta_0 > \cdots > \theta_d$. Let $\theta_i$ be the largest eigenvalue such that $\theta_i \leq -1$. If $\theta_i \geq -2$ and $\theta_{i-1} \leq 0$, then the Hoffman-type bound tells us
    \begin{equation*}
        \chi_2(G) \geq
        \begin{cases}
            k+1, &\text{ if } \theta_i = -1 \text{ and } \abs{V} = 0 \text{ (mod }k+1),\\
            k+2, &\text{ else.}
        \end{cases} 
    \end{equation*}
\end{lemma}
\begin{proof}
    The proof works mostly the same as that of \cite[Corollary 29]{ANR2024}, where we show that, if $-1>\theta_i \geq -2$ and $\theta_{i-1} \leq 0$, then the bound from Corollary \ref{cor:ratio_chiprime_2_reg} is at most $k+2$ and it is strictly more than $k+1$. 
    
    We strengthen this result, by using the observation that if $\theta_i = -1$ and $\abs{V} \neq 0 \text{ (mod }k+1)$, then
    \begin{equation*}
        \chi_2(G) \geq \frac{\abs{V}}{\left\lfloor\abs{V}\frac{k-\theta_{i-1}}{(k+1)(k-\theta_{i-1})}\right\rfloor} = \frac{\abs{V}}{\left\lfloor\frac{\abs{V}}{k+1}\right\rfloor} > k+1
    \end{equation*}
    and hence $\chi_2(G) \geq k+2$, since $\chi_2(G)$ is an integer.
\end{proof}

We will now investigate the spectrum of $L(C_q \bowtie C_{q'})$, and show that it has eigenvalues such that we can apply Lemma \ref{lem:chi2_ratio_behaviour}, allowing us to conclude the Hoffman-type bound will give a lower bound of $6$. 

As a direct result of Proposition \ref{prop:balbiprod_adjacency}, we have the following characterization of the eigenvalues of $C_{2r} \bowtie C_{2r'}$.

\begin{corollary}\label{cor:balbiprod_eig_eq}
    Let $C_{2r},C_{2r'}$ be two even cycles. Let $\begin{pmatrix}
        O & A \\
        A^\top & O
    \end{pmatrix}$ be the adjacency matrix of $C_{2r}$. Then $\lambda$ is an eigenvalue of $C_{2r} \bowtie C_{2r'}$ if and only if there exist vectors $u_i,v_i$ of length $r$ for $i \in [r']$, at least one of which is nonzero, which satisfy:
    \begin{align*}
        Au_i + u_{i+1} &= \lambda v_i \text{ for } i =1,\hdots, r'-1, \\
        Au_{r'} + u_1 &= \lambda v_{r'}, \\
        A^\top v_i + v_{i-1} &= \lambda u_i \text{ for } i =2, \hdots, r', \\
        A^\top v_1 + v_{r'} &= \lambda u_1.
    \end{align*}
\end{corollary}
\begin{proof}
    This follows directly from letting $\tilde{v} = (v_1,u_1,v_2,u_2,\hdots,v_{r'},u_{r'})$ and solving $\tilde{A}\tilde{v}$ using the expression for $\tilde{A}$ derived in Proposition \ref{prop:balbiprod_adjacency}.
\end{proof}

\begin{lemma}\label{lem:balbiprod_-1_eig}
     Let $C_{q},C_{q'}$ be two even cycles  with $q,q' = 0 \text{ (mod }4)$. Then $-1$ is an eigenvalue of $C_q \bowtie C_{q'}$.
\end{lemma}
\begin{proof}
    Let
    \begin{equation*}
    u_i = -v_i = (1,-1,\hdots,1,-1).    
    \end{equation*}
    Note that $Au_i = A^\top v_i = 0$, and hence, with $\lambda = -1$, the equations from Proposition \ref{cor:balbiprod_eig_eq} are satisfied, meaning $-1$ is indeed an eigenvalue.
\end{proof}

We now have all the preliminary results need to prove Proposition \ref{prop:balbiprod_tight}.

\begin{proof}[Proof of Proposition \ref{prop:balbiprod_tight}]
    Note that by Lemma \ref{lem:balbiprod_colour} we know a bound is tight if it tells us $\chi'_2(H) \geq 6$. By Corollary \ref{cor:LG_spec}, we can fully characterize the spectrum of $L(H)$ in terms of the spectrum of $H$. In this case in particular, if $H$ has ordered eigenvalues $\theta_0 > \cdots > \theta_d$, then $L(H)$ has eigenvalues $\theta_0' > \cdots > \theta'_d$, with $\theta'_i = \theta_i + 1$. (Using that $H$ is $3$-regular by \cite[Proposition 2.3]{kang_distance_2017}).

    Now we can use Lemma \ref{lem:balbiprod_-1_eig} in combination with the fact that $H$ is a cubic bipartite graph to find that $L(H)$ has eigenvalues in both the interval $[0,-1)$ and $[-1,-2]$. Thus Lemma \ref{lem:chi2_ratio_behaviour} applies to $L(H)$. Finally, we use the observation $\abs{E} =\frac{3qq'}{2}$.
\end{proof}

We should note that, computationally, Proposition \ref{prop:balbiprod_tight} seems to hold under similar conditions when we do not restrict $q,q' = 0 \text{ (mod }4)$, but we cannot show tightness in this case until a valid colouring with $6$ colours is constructed in general.
\subsubsection{An infinite family of graphs for which Corollary \ref{cor:ratio_chiprime_3} is tight} \label{sec:chi3_tight}

A special case of a balanced bipartite product of cycles, are the graphs $GM(k) = C_4 \bowtie C_{2k}$. These are so called \emph{Guo-Mohar graphs}, which were first introduced in \cite{guo_large_2014}. They recently came to prominence in \cite{guo_cubic_2026}, where they played a large role in the classification of all cubic graphs with no eigenvalues in the interval $(-1,1)$. In this section, we will prove Proposition \ref{prop:GM_tight}, which shows that Corollary \ref{cor:ratio_chiprime_3} gives a tight bound for the Guo-Mohar graphs $GM(2k)$.

\begin{proposition}\label{prop:GM_tight}
     The Hoffman-type bound for $\chi'_3(GM(2k))$ (Corollary \ref{cor:ratio_chiprime_3}) is tight for the Guo-Mohar graphs $GM(2k)$.
\end{proposition}

As seen in \cite{guo_cubic_2026}, the Guo-Mohar graphs are one of only two infinite families of cubic graphs with no eigenvalues in $(-1,1)$. It is this property of their spectrum that will allow us to show Corollary \ref{cor:ratio_chiprime_3} is tight for $GM(2k)$.
\begin{theorem}{(\cite[2.1 Theorem]{guo_large_2014})}\label{thm:GM_spec}
    For $k \geq 2$, the spectrum of $GM(k)$ consists of eigenvalues $\pm 1$, each with multiplicity $k$, and the values
    \begin{equation*}
        \pm \sqrt{5+4\cos(2\pi j/k)} 
    \end{equation*}
    for $0 \leq j \leq k-1$. In particular, $GM(k)$ has no eigenvalues in the range $(-1,1)$.
\end{theorem}

Before we compute the value of our bound from Corollary \ref{cor:ratio_chiprime_3} for $GM(2k)$, we will first provide an upper bound on $\chi'_3(GM(2k))$ by colouring it using $12$ colours.
\begin{lemma}\label{lem:balbiprod_colour_3}
    Let $GM(2k)$ be a Guo-Mohar graph. Then $\chi'_3(GM(2k)) \leq 12$.
\end{lemma}
\begin{proof}
    We colour the edges using $12$ colours as follows: 
        \begin{align*}
        c((v_i^1,v_j^2),(u_i^1,u_j^2)) &= i + 6j \text{ (mod }12)\\
        c((v_2^1,v_j^2),(u_1^1,u_j^2)) &= 3 + 6j \text{ (mod }12)\\
        c((v_1^1,v_j^2),(u_2^1,u_j^2)) &= 4 + 6j \text{ (mod }12)\\
        c((v_i^1,v_{j+1}^2),(u_i^1,u_j^2)) &= 4+i+6j \text{ (mod }12)
    \end{align*}
    where we use abuse of notation $v^2_{k+1} = v_1^2$. See Figure \ref{fig:12_colour_GMk} for an example (where dashed lines represent different colours, e.g. dashed red and red are distinct colours).
\end{proof}

\begin{figure}[H]
    \centering
    \begin{tikzpicture}
        \node[fill,circle] (u1) at (-1,0){};
        \node[fill,circle] (u2) at (0,0){};
        \node[fill,circle] (u3) at (-1,-1){};
        \node[fill,circle] (u4) at (0,-1){};
        \node[fill,circle] (v1) at (2,0){};
        \node[fill,circle] (v2) at (3,0){};
        \node[fill,circle] (v3) at (2,-1){};
        \node[fill,circle] (v4) at (3,-1){};
        \node[fill] (W1) at (-3,-.5){};
        \node[fill] (W2) at (5,-.5){};
        \draw[thick, color=red] (u1)--(u2);
        \draw[thick, color=blue] (u2)--(u3);
        \draw[thick, color=green] (u3)--(u4);
        \draw[thick, color=orange] (u4)--(u1);
        \draw[thick, dashed, color=red] (v1)--(v2);
        \draw[thick, dashed, color=orange] (v2)--(v3);
        \draw[thick, dashed, color=green] (v3)--(v4);
        \draw[thick, dashed, color=blue] (v4)--(v1);
        \draw[thick, color=cyan] (u2)--(v1);
        \draw[thick, dashed, color=cyan] (u4)--(v3);
        \draw[thick, color=magenta] (W1)--(u1);
        \draw[thick, dashed, color=magenta] (W1)--(u3);
        \draw[thick, color=magenta] (W2)--(v2);
        \draw[thick, dashed, color=magenta] (W2)--(v4);
    \end{tikzpicture}
    \caption{An example of how to $12$-colour part of $GM(2k)$. Dashed and non-dashed lines represent different colours. The square nodes represent the remainder of the graph.}
    \label{fig:12_colour_GMk}
\end{figure}

We are now ready to state the proof of Proposition \ref{prop:GM_tight}.

\begin{proof}[Proof of Proposition \ref{prop:GM_tight}]
    In order to obtain a bound from Corollary \ref{cor:ratio_chiprime_3}, we first need to calculate $\Delta'_3$. Recall, this is (twice) the number of triangles an edge $e \in E$ is in, in $L(GM(2k))$. Since $GM(2k)$ is triangle free, the only such triangles in $L(GM(2k))$ will be of the form $(u,v),(u,w),(u,x)$, i.e.\ pairs of vertices $\{w,x\}$ adjacent to one of the endpoints of the edge $(u,v)$. Since $GM(2k)$ is $3$-regular, each vertex has exactly one such pair, so one triangle. Thus for both vertices in an edge, we have one triangle, but since we need to count them twice (once for each direction it can be traversed), we have $\Delta'_3 = 4$.
    
    Next, by combining Theorem \ref{thm:GM_spec} with Lemma \ref{lem:LG_reg}, we find that $L(GM(2k))$ has largest eigenvalue $\theta'_0 = 4$, smallest eigenvalue $\theta'_d = -2$ and it has no eigenvalues in the range $(0,2)$, but does have $0$ and $2$ themselves as eigenvalues. Recall, that $\theta'_s$ is the largest eigenvalue $$\theta'_s \leq - \frac{\theta'^2_0 + \theta'_0\theta'_{d'}-\Delta'_3}{\theta'_0 (\theta'_{d'} + 1)} = -\frac{16-12-4}{-4} = 1.$$ But since $L(GM(2k))$ has no eigenvalues in $(0,2)$, we know $\theta'_s = 0$ and $\theta'_{s-1} = 2$.
    
    Plugging all this in gives us the bound:
    \begin{align*}
        \chi'_3(GM(2k)) &\geq \frac{(\theta'_0 - \theta'_s)(\theta'_0 - \theta'_{s-1})(\theta'_0 - \theta'_{d'})}{\Delta'_3 - \theta'_0(\theta'_s + \theta'_{s-1} + \theta'_{d'}) - \theta'_s\theta'_{s-1} \theta'_{d'}} \\
        &=\frac{(4 - 0)(4 - 2)(4+2)}{4-4(0+2-2)-0\cdot2\cdot (-2)} \\
        &= \frac{4\cdot 2\cdot 6}{4} = 12,
    \end{align*}
which we know is tight by Lemma \ref{lem:balbiprod_colour_3}.
\end{proof}

\subsection{Computational performance}
Next we investigate how the bounds perform for a larger variety of graphs which are built into Sagemath. The results can be found in the Appendix in Tables \ref{tab:bound_sim_2} and \ref{tab:bound_sim_3}. In Table \ref{tab:bound_sim_2} the performance of the Hoffman-type lower bound (Corollary \ref{cor:ratio_chiprime_2} /\ref{cor:ratio_chiprime_2_reg}) and the first inertial-type lower bound (Theorem \ref{thm:1st_inertial_chiprime}), as well as the induced Wilf upper bound (Corollary \ref{cor:wilf}) are compared to the exact value of $\chi'_2(G)$. Analogously, in Table \ref{tab:bound_sim_3} we compare the Hoffman-type lower bound (Corollary \ref{cor:ratio_chiprime_3} / MILP implementation of Theorem \ref{thm:ratio_chiprime}), the first inertial-type lower bound (Theorem \ref{thm:1st_inertial_chiprime}) and the induced Wilf upper bound (Corollary \ref{cor:wilf}) to the exact value of $\chi'_3(G)$.
    
The exact values for $\chi'_2(G)$ and $\chi'_3(G)$ were computed using SAT, by implementing the formulation for $\chi(G)$ from \cite[Section 2.2]{faber_sat_2024}, and using Kissat in Sagemath to solve it. In both Table \ref{tab:bound_sim_2} and Table \ref{tab:bound_sim_3}, if for some graph it takes more than 5 minutes to compute the exact parameter value or to solve the MILP or LP implementation of Theorem \ref{thm:1st_inertial_chiprime} or Theorem \ref{thm:ratio_chiprime} respectively, the corresponding entry in the table is denoted by ``time'' instead. Graph names are in boldface if one of our bounds is tight.  

While $\chi_t'(G)$ takes too long to compute for a sizable number of graphs, for those that were computable, we see that the bounds from Corollaries \ref{cor:ratio_chiprime_2_reg} and \ref{cor:ratio_chiprime_3} perform well. Indeed, for a fair few graphs we find that the bounds are tight, and for others it gets close to the exact parameter value. 

The MILP for first inertial-type bound has a more mixed performance. In Tables \ref{tab:bound_sim_2} and \ref{tab:bound_sim_3}, we see that the first inertial-type bound never outperforms the Hoffman-type bounds. However, we have found (computationally) numerous examples of smaller graphs (less than $10$ vertices), for which the first inertial-type bound is tight, whereas the Hoffman-type bounds are not. Take for instance $\chi'_2(P_4)$, where $P_4$ is the path graph on $4$ vertices. The first inertial-type bound gives a lower bound of $3$ (with polynomial $x^2 -\sqrt{2}x$), whereas the Hoffman-type bound only gives a lower bound of $2$.

\section{On the distance-\texorpdfstring{$t$}{t} Erdős-Nešetřil problem}\label{sec:erdos_nesetril}

    In 1985, Erdős and Nešetřil  asked what the best upper bound on $\chi_2'(G)$ is in terms of $\Delta$. This spawned the earlier seen Conjecture \ref{con:erdos_nesetril}.

    Much work has been done on finding upper bounds of the form $\chi'_2(G) \leq b_2\Delta^2$, where $b_2 \leq 2$ is a positive constant (or equivalently $\chi'_2(G) \leq (2-\varepsilon_2)\Delta^2$), for large enough $\Delta$. We already saw Erdős and Nešetřil conjectured $b_2$ could be as low as $\frac{5}{4}$, but even whether or not $b_2 <2$ was unknown. The only thing that was known was that $b_2$ could not be less than $\frac{5}{4}$. In 1997, Molloy and Reed \cite{molloy_bound_1997} managed to prove $b_2<2$ by using the probabilistic method to show $\chi'_2(G) \leq 1.998\Delta^2$. This bound was further tightened to $\chi'_2(G) \leq 1.93\Delta^2$ by Bruhn and Joos \cite{bruhn_stronger_2015}, and to $\chi'_2(G) \leq 1.772\Delta^2$ by Hurley et al.\ \cite{hurley_improved_2021}. Most recently, Kang \cite{Kang2024} announced that a bound of $\chi'_2(G) \leq 1.730\Delta^2$ was obtained, though this is not yet published. Additionally, much research has been done into $\chi_2'(G)$ for specific classes of graphs, see the introduction of \cite{wang_strong_2018} where some of this progress is laid out.

    For the general distance-$t$ chromatic index, a trivial upper bound is $\chi_t'(G) \leq 2 \sum_{j=1}^t(\Delta-1)^j+1$. Just like for the $t=2$ case, bounds of the form $\chi'_t(G) \leq b\Delta^t$ (or equivalently $\chi'_t(G) \leq (1-\varepsilon)\Delta^t)$ have received much attention, spurred on by the fact that $\chi_t'(G) = \Omega(\Delta^t)$ \cite{kang_distance_2012}. In 2014 Kaiser and Kang \cite{kaiser_distance-t_2014} provided the first such bound by using the probabilistic method, which showed that for large enough $\Delta$ we have $\chi_t'(G) \leq 1.99992\Delta^t$. In addition to this bound, Kaiser and Kang also derived some results on the behaviour of $\chi'_t(G)$ given $G$ has a certain girth. An improvement to the bound was recently made by Cambie et al.\ \cite{cambie_maximizing_2022}, where they showed that for large enough $\Delta$ we have $\chi_t'(G) \leq 1.941\Delta^t$.

    In this section, we apply our spectral approach to the Erdős-Nešetřil conjecture. First, we show why Wilf's bound (Theorem \ref{thm:wilf}) cannot be used to strengthen the known results. Then, we look at the Hoffman-type bound (Corollary \ref{cor:ratio_chiprime_2}), and derive some conditions on the spectrum that a potential counterexample from our approach would have to satisfy.
\subsection{Wilf's bound and the Erdős-Nešetřil conjecture}

    Here we look at a classical eigenvalue upper bound on the chromatic number of a graph (Wilf's bound) and how we can use it to obtain a bound on $\chi_t'(G)$. Then we investigate how this bound on $\chi_t'(G)$ performs for the distance-$t$ Erd\H{o}s-Nešetřil problem.
    


\begin{theorem}[Wilf's bound \cite{wilf_eigenvalues_1967}]\label{thm:wilf}
    Let $G$ be a graph with largest adjacency eigenvalue $\lambda_1$. Then
    $\chi(G) \leq 1 + \lambda_1$.
\end{theorem}

Theorem \ref{thm:wilf} can be applied to $L(G)^t$ to readily obtain a bound on $\chi'_t(G)$ by using \eqref{eq:chrom_index_number_relation}.
\begin{corollary}\label{cor:wilf}
    Let $G$ be a graph, let $\tilde \lambda_1$ be the largest adjacency eigenvalue of $L(G)^t$. Then
       $\chi'_t(G) \leq 1 + \tilde \lambda_1$.
\end{corollary}

Rather than Corollary \ref{cor:wilf}, which uses the spectrum of $L(G)^t$, it would be desirable to obtain a Wilf-like bound which uses the spectrum of $L(G)$ itself,  as we did for the lower bounds in Section \ref{sec:spec_chiprimet}. However, below we show that one  runs into issues due to the distance $t$ analog of the so-called critical graphs. A graph $G$ is called \emph{critical} if any proper subgraph $H$ satisfies $\chi(G) > \chi(H)$. It is well known that for a critical graph $G$, all of its vertices satisfy $\chi(G) - d(v) \leq 1$. This property is exploited in the proof of Wilf's bound. If we want an analogous bound for $\chi_t(G)$, we would need this property to extend to the $t$-chromatic number. To this end, we call a graph \emph{$t$-critical} if any proper subgraph $H$ satisfies $\chi_t(G) > \chi_t(H)$, and furthermore, we denote by $d_t(v)$ the amount of vertices within distance $t$ of a vertex $v \in V(G)$. Analogous to the $t=1$ case, we would want a bound like $\chi_t(G) - d_t(v) \leq 1$ or similar to hold for $t$-critical graphs. However, this is not the case as the following result shows.

\begin{proposition}\label{prop:t-crit}
Let $t>1$ be a fixed integer. Then for any $M \in \mathbb{N}$, there exists a $t$-critical graph $G$ with vertex $v \in V(G)$ such that $\chi_t(G) - d_t(v) > M$. In other words, the quantity $\chi_t(G) - d_t(v)$ can be arbitrarily large for $t$-critical graphs.
\end{proposition}
\begin{proof}
Consider two paths $Q_1,Q_2$ of order $t$. Denote the vertices of these paths $u_1,\hdots,u_t$ and $v_1,\hdots,v_t$. Now connect both $u_1$ and $v_t$ to every vertex in a complete graph $K_n$, and connect both $u_t$ and $v_1$ to a lone vertex $w'$. Further, connect $u_i$ and $v_i$ to a path of order $t-1$, $Q'_i$. Denote the vertices of $Q'_i$ by $w_{i,1}, \hdots, w_{i,t-1}$, with $w_{i,1}$ connected to $u_i$. Call the resulting graph $G_t$. See the left-most graph in Figure \ref{fig:2-crit} for an example when $t=2$.

Let us first reason why $G_t$ is $t$-critical. For this, we will have to show that removing any vertex decreases the $t$-chromatic number. The way we do this, is by claiming certain types of vertices are within distance $t$ of each other, and thus must share a colour in the original graph.
\begin{claim}
    For all $i\neq j$, $d(u_i,v_j) \leq t + 1 - \abs{i-j} \leq t$. And if $i = j$, then $d(u_i,v_j) \leq t$. 
\end{claim}
\begin{proof}
    If $i=j$, then the path $(u_i, Q_i', v_i)$ is of length exactly $t$. Instead, if $i \neq j$, then consider the cycle of length $2t+2$: $(u_1,\hdots, u_t, w', v_1,\hdots,v_t, K_n,u_1)$. This gives us exactly the upper bound on the distances we want.
\end{proof}
\begin{claim}
    For all $i,j,\ell$, $d(u_i,w_{j,\ell}) \leq t$ and $d(v_i,w_{j,\ell}) \leq t$.
\end{claim}
\begin{proof}
    Using Claim 1, 
    \begin{align*}
        d(u_i,w_{j,\ell}) &\leq d(u_i,u_j) + d(u_j,w_{j,\ell}) < \abs{i-j} + \ell, \\
        d(u_i,w_{j,\ell}) &\leq d(u_i,v_j) + d(v_j,w_{j,\ell}) \leq t+1-\abs{i-j} + t-\ell.
    \end{align*}
    If $\abs{i-j}+\ell \leq t$, then the first inequality works, if $\abs{i+j}+\ell \geq t+1$, then the second inequality works. The inequality $d(v_i,w_{j,\ell}) \leq t$ can be proven analogously.
\end{proof}
\begin{claim}
    $w'$ is within distance $t$ of all vertices in $G_t$, except those in $K_n$. Similarly, the vertices in $K_n$ are within distance $t$ of all vertices in $G_t$, except $w'$.
\end{claim}
\begin{proof}
    We will show it for $w'$, the $K_n$ case works analogously. First, $d(w',u_i) \leq t$ and $d(w',v_i) \leq t$ follow from the cycle presented in the proof of Claim 1. For $d(w',w_{i,\ell})$, we find
    \begin{align*}
        d(w',w_{i,\ell}) &\leq d(w', u_i) + d(u_i, w_{i,\ell}) = i+\ell, \\
        d(w',w_{i,\ell}) &\leq d(w', v_i) + d(v_i, w_{i,\ell}) = t+1-i + t-\ell.
    \end{align*}
    If $i+l \leq t$, then the first inequality works, if $i+l \geq t+1$, the second works.
\end{proof}
Now let us consider an optimal colouring of $G_t$ using $\chi_t(G_t)$ colours. Claims 1 and 3 tell us $Q_1,Q_2$ and $K_n$ are all within distance $t$ and hence all vertices in $Q_1,Q_2$ and $K_n$ must have their own unique colour, and w.l.o.g. we can assume $w'$ shares its colour with a vertex in $K_n$. Now we will show that removing a vertex from $G_t$, must decrease the distance $t$-chromatic number. Also see Figures \ref{fig:3-crit} and \ref{fig:2-crit} for two small examples.
\begin{itemize}
    \item Suppose we remove a vertex of $Q_1$ or $Q_2$. Then, by Claims 1,2 and 3, this vertex was within distance $t$ of all other vertices of $G_t$, so we know no vertex had the same colour. Hence, by removing this vertex, we have used one fewer colour, thus the $t$-chromatic number has decreased.
    \item Suppose we remove a vertex of $K_n$. W.l.o.g. we can assume this vertex did not share its colour with $w'$ (otherwise we switch $w'$'s colour to match a different vertex in $K_n$). Then, by the same reasoning as in the previous case, we can colour the new subgraph using one fewer colour.
    \item Suppose we remove $w'$. Then $d(u_t,v_1) >t$ in the new graph, and hence $u_t$ and $v_1$ can share a colour, reducing the total amount of colours needed, and thus this subgraph has a smaller $t$-chromatic number.
    \item Suppose we remove a vertex of some $Q_i'$. Then $d(u_i,v_i)=t+1$, and hence $u_i$ and $v_i$ can share a colour. We once again find that this subgraph has a smaller $t$-chromatic number.
\end{itemize}
In conclusion, $G_t$ is $t$-critical. Finally, we note that the $d_t(w') = t(t+1)$, but $\chi_t(G_t) > n$, hence $\chi_t(G_t) - d_t(w')> n-t(t+1)$ can be arbitrarily large.
\end{proof}

\begin{figure}[H]
    \centering
    \begin{tikzpicture}
        \node[fill,circle, label=$u_1$] (u1) at (0,0){};
        \node[fill,circle, label=$u_2$] (u2) at (0,-2){};
        \node[fill,circle, label=$u_3$] (u3) at (0,-4){};
        \node[fill,circle, label=$v_1$] (v1) at (3,0){};
        \node[fill,circle, label=$v_2$] (v2) at (3,-2){};
        \node[fill,circle, label=$v_3$] (v3) at (3,-4){};
        \node[fill,circle, label=$w'$] (w') at (0,2){};
        \node[fill,circle, label=$w_{1,1}$] (w11) at (1,0){};
        \node[fill,circle, label=$w_{1,2}$] (w12) at (2,0){};
        \node[fill,circle, label=$w_{2,1}$] (w21) at (1,-2){};
        \node[fill,circle, label=$w_{2,2}$] (w22) at (2,-2){};
        \node[fill,circle, label=$w_{3,1}$] (w31) at (1,-4){};
        \node[fill,circle, label=$w_{3,2}$] (w32) at (2,-4){};
        \node[fill,label=$K_n$] (k1) at (3,2){};
        \draw (v1)--(v2)--(v3)--(w32)--(w31)--(u3)--(u2)--(u1)--(w11)--(w12)--(v1);
        \draw (u2)--(w21)--(w22)--(v2);
        \draw (w')  to[bend right=45] (u3);
        \draw (u1)--(k1);
        \draw (v1)--(w');
        \draw (k1) to[bend left=45] (v3);
    \end{tikzpicture}
    \caption{The graph $G_3$ as defined in the proof of Proposition \ref{prop:t-crit}.}
    \label{fig:3-crit}
\end{figure}

\begin{figure}[H]
    \centering
    \subfloat{\begin{tikzpicture}
        \node[fill,circle, label=$u_1$, color=red] (u1) at (0,0){};
        \node[fill,circle, label=$u_2$, color=blue] (u2) at (0,-2){};
        \node[fill,circle, label=$v_1$, color=orange] (v1) at (2,0){};
        \node[fill,circle, label=$v_2$, color=cyan] (v2) at (2,-2){};
        \node[fill,circle, label=$w'$, color=magenta] (w') at (0,2){};
        \node[fill,circle, label=$w_{1,1}$, color=green] (w1) at (1,0){};
        \node[fill,circle, label=$w_{2,1}$, color=green] (w2) at (1,-2){};
        \node[fill,label=$K_n$,color=magenta] (k1) at (2,2){};
        \draw (v1)--(v2)--(w2)--(u2)--(u1)--(w1)--(v1)--(w');
        \draw (w')  to[bend right=45] (u2);
        \draw (u1)--(k1);
        \draw (k1) to[bend left=45] (v2);
    \end{tikzpicture}}
    \subfloat{\begin{tikzpicture}
        \node[fill,circle, label=$u_1$, color=red] (u1) at (0,0){};
        \node[fill,circle, label=$u_2$, color=blue] (u2) at (0,-2){};
        \node[fill,circle, label=$v_1$, color=blue] (v1) at (2,0){};
        \node[fill,circle, label=$v_2$, color=cyan] (v2) at (2,-2){};
        \node[fill,circle, label=$w_{1,1}$, color=green] (w1) at (1,0){};
        \node[fill,circle, label=$w_{2,1}$, color=green] (w2) at (1,-2){};
        \node[fill,label=$K_n$,color=magenta] (k1) at (2,2){};
        \draw (v1)--(v2)--(w2)--(u2)--(u1)--(w1)--(v1);
        \draw (u1)--(k1);
        \draw (k1) to[bend left=45] (v2);
    \end{tikzpicture}}
        \subfloat{\begin{tikzpicture}
        \node[fill,circle, label=$u_1$, color=red] (u1) at (0,0){};
        \node[fill,circle, label=$u_2$, color=blue] (u2) at (0,-2){};
        \node[fill,circle, label=$v_1$, color=orange] (v1) at (2,0){};
        \node[fill,circle, label=$v_2$, color=blue] (v2) at (2,-2){};
        \node[fill,circle, label=$w'$, color=magenta] (w') at (0,2){};
        \node[fill,circle, label=$w_{1,1}$, color=green] (w1) at (1,0){};
        \node[fill,label=$K_n$,color=magenta] (k1) at (2,2){};
        \draw (v1)--(v2);
        \draw (u2)--(u1)--(w1)--(v1)--(w');
        \draw (w')  to[bend right=45] (u2);
        \draw (u1)--(k1);
        \draw (k1) to[bend left=45] (v2);
    \end{tikzpicture}}
    \caption{A visual proof of why $G_2$ as defined in the proof of Proposition \ref{prop:t-crit} is $2$-critical.}
    \label{fig:2-crit}
\end{figure}

Despite the fact that we cannot obtain a Wilf-like bound tailored specifically for $\chi'_t(G)$ due to Proposition \ref{prop:t-crit}, we can still study whether Corollary \ref{cor:wilf} can tell us anything regarding Conjecture \ref{con:erdos_nesetril}. In particular, we are interested in whether Corollary \ref{cor:wilf} can be used to derive a competitive bound of the form $\chi'_t(G) \leq b_t\Delta^t$. However, for $t=2$, we can show quite easily that the best we can hope to do is $b_2 = 2$, which is significantly worse than the previously derived bound of $b_2 = 1.730$ \cite{Kang2024}. This is because we can fully characterize the behaviour of Wilf's bound for $L(G)^2$ when $G$ is a strongly regular graph.
\begin{lemma}\label{lem:SRG_Wilf}
    Let $G$ be a SRG($n,k,\lambda,\mu)$. Then $L(G)^2$ is regular with valency
    \begin{equation*}
        2k^2+(1+\lambda)(\mu-1)-k(1+\lambda+\mu)
    \end{equation*}
    and Wilf's bound becomes
    \begin{equation*}
        \chi_2'(G) \leq 2k^2+(1+\lambda)(\mu-1)-k(1+\lambda+\mu)+1.
    \end{equation*}
\end{lemma}
\begin{proof}
    Consider an arbitrary edge $(u,v) \in E(G)$. We want to know how many vertices are within distance $2$ of $(u,v)$ in $L(G)$. We already know by Lemma \ref{lem:LG_reg} that there are $2(k-1)$ vertices adjacent to $(u,v)$. Now let us count the vertices at exactly distance $2$.

    \begin{itemize}
        \item First, we count the vertices of the form $(w,x)$ where $(u,w) \in E$. There are $k-1$ choices for $w$, and exactly $\lambda$ of these will also be adjacent to $u$. For these $\lambda$ choices of $w$, we have $k-2$ choices of $x$. The remaining $k-1-\lambda$ choices of $w$ have $k-1$ choices of $x$. Hence in total we have
        \begin{equation*}
            (k-1-\lambda)(k-1) + \lambda(k-2),
        \end{equation*}
        vertices at distance $2$ of this type.
        \item Next, we want to count the vertices of the form $(w,x)$ where $(v,w) \in E$. There are again $k-1$ choices for $w$, but $\lambda$ of these were already considered in the previous case. For the remaining $(k-1-\lambda)$ choices of $w$, we have $k-1$ choices of $x$. However, exactly $\mu-1$ of these choices are also adjacent to $v$ (i.e.\ $(x,v) \in E(G)$), so they were already counted in the previous case. Thus we only have an additional
        \begin{equation*}
            (k-1-\lambda)(k-1-(\mu-1)) = (k-1-\lambda)(k-\mu)
        \end{equation*}
        vertices at distance $2$.
    \end{itemize}
    In total, we find 
    \begin{align*}
        &2(k-1) + (k-1-\lambda)(k-1) + \lambda(k-2) + (k-1-\lambda)(k-\mu)\\
        = &2k^2+(1+\lambda)(\mu-1)-k(1+\lambda+\mu).
    \end{align*}
    vertices within distance $2$. Since this holds for arbitrary $(u,v)$, we find $L(G)^2$ is regular and thus its largest eigenvalue is equal to the maximum degree. Hence, Wilf's bound is
    \begin{align*}
   &     \chi_2'(G) \leq 2k^2+(1+\lambda)(\mu-1)-k(1+\lambda+\mu)+1.
 \qedhere   \end{align*}
\end{proof}

As there are a large amount of strongly regular graphs known and documented (see \cite{brouwer_distance_regular_1989} and \cite{brouwer_parameters_nodate}), Lemma \ref{lem:SRG_Wilf} allows us to investigate the behaviour of Wilf's bound on a large class of graphs. In particular, we are able to find an infinite family of strongly regular graphs of Lie type, for which Wilf's bound gets arbitrarily close to $\chi'_2 (G)\leq 2\Delta^2$.

\begin{corollary}
    For every $\varepsilon>0$, there exists a graph $G$ of maximum degree $\Delta$ such that Corollary \ref{cor:wilf} gives a bound worse than
    \begin{equation*}
        \chi'_2(G) \leq (2-\varepsilon)\Delta^2.
    \end{equation*}
\end{corollary}
\begin{proof}
    We use the Lie type SRG construction $E_6(q)$ (see \cite[Table 10.8]{brouwer_distance_regular_1989}), which has parameters
    \begin{align*}
        n &= \frac{(q^{12}-1)(q^9-1)}{(q^4-1)(q-1)}, \\
        k &= \frac{q(q^3-1)(q^8-1)}{q-1}, \\
        \lambda &= k - 1 - \frac{q^7(q^5-1)}{q-1}, \\
        \mu &= \frac{(q^3+1)(q^4-1)}{q-1}.
    \end{align*}
    Plugging these values into Lemma \ref{lem:SRG_Wilf} gives us the bound
    \begin{align*}
        B(q) = 2k^2+(1+\lambda)(\mu-1)-k(1+\lambda+\mu).
    \end{align*}

    It is not hard to see that $k = O(q^{11}), \lambda = O(q^8)$ and $\mu = O(q^6)$, hence
    \begin{equation*}
        \lim_{q \rightarrow \infty} \frac{B(q)}{k^2} = 2.
    \end{equation*}

    Thus, as we take $q$ larger, the Wilf bound approaches $2\Delta^2$ for $E_6(q)$.
\end{proof}

Thus, we conclude that Wilf's bound will not help in making progress towards Conjecture \ref{con:erdos_nesetril}. However, we cannot currently rule out that Wilf's bound could be applied to the $t \geq 3$ Erdős-Nešetřil problem.

\subsection{Hoffman-type bound and the Erdős-Nešetřil conjecture}
In Section \ref{sec:spec_chiprimet} we derived a spectral Hoffman-type lower bound on $\chi_2'(G)$ (Corollary \ref{cor:ratio_chiprime_2}). In theory, if there is a graph $G$ for which the bound from Corollary \ref{cor:ratio_chiprime_2} gives a lower bound larger than $\frac{5\Delta^2}{4}$, this would disprove Conjecture \ref{con:erdos_nesetril}. In this section, we show that, if such a graph were to exist, then it (or its line graph) must have a large subinterval $I \subset [\theta_d,\theta_0]$ such that none of the eigenvalues are contained in $I$. In particular, the length of $I$ (i.e. the distance between $I$'s suprememum and infimum) must be of order $O(\sqrt{\Delta})$ or in some cases even $O(\Delta)$.

\begin{proposition}\label{prop:hoffman_EN_regular}
    Let $G$ be a $k$-regular graph, $k \geq 2$, for which the Hoffman-type bound (Corollary \ref{cor:ratio_chiprime_2_reg}) gives a lower bound larger than $ck^2$ for some $c>1$. Then $G$ has no eigenvalues in the interval
    \begin{align*}
        &\left(-1,-2+\frac{4\sqrt{8c^2-6c}}{4c-1}\right) &\text{ if } k=2, \\
        &\left(-k+1,\frac{ck-k}{ck-1}\right) &\text{ if } k \geq 3.
    \end{align*}
    In particular, for $k \geq 3$, there is an interval $I \subset [-k,k]$ of size $\abs{I} \geq k$ which contains no eigenvalues of $G$.
\end{proposition}
\begin{proof}        
    For a $k$-regular graph $G$ the bound from Corollary \ref{cor:ratio_chiprime_2_reg} has the form
    \begin{equation*}
        B(\theta_{i-1}') =\frac{(2k-2-\theta'_i)(2k-2-\theta'_{i-1})}{2k-2+\theta'_i\theta'_{i-1}}.
    \end{equation*}

    By solving $B(\theta_{i-1}') \geq ck^2$ we obtain the inequality
    \begin{equation*}
        \theta'_{i-1} \geq \frac{(2k-2)(ck^2+\theta'_i-2k+2)}{-\theta'_i(ck^2-1)-2k+2},
    \end{equation*}
    and hence
    \begin{equation*}
        \theta'_{i-1}-\theta'_i \geq \frac{(2k-2)(ck^2+\theta'_i-2k+2)}{-\theta'_i(ck^2-1)-2k+2} - \theta'_i=:d(\theta'_i).
    \end{equation*}
    By finding the solutions $d'(\theta_i')=0$ and using that since $L(G)$ is a line graph, all its eigenvalues are larger than $-2$ (by e.g. the equivalence with the PSD signless Laplacian as seen in \cite[Proposition 1.4.1]{brouwer_spectra_2012}), we can find that  $d(\theta'_i)$ has its minimum at
    \begin{equation*}
        \mu := \max\left\{-2,\frac{-(2k-2)- \sqrt{2ck^2(k-1)(1+k(ck-2))}}{ck^2-1}\right\}.
    \end{equation*} 
    For $k=2$, it becomes     
    $$\mu = \frac{-(2k-2)- \sqrt{2ck^2(k-1)(1+k(ck-2))}}{ck^2-1}.$$ 
    Hence, we obtain the bound
    \begin{equation*}
        \theta'_{i-1}-\theta'_i \geq \frac{2\sqrt{2ck^2(k-1)(1+k(ck-2))}}{ck^2-1} = \frac{4\sqrt{8c^2-6c}}{4c-1}.
    \end{equation*}
    Once again using that $\theta_i \geq -2$, we find that $L(G)$ must have no eigenvalues in the interval
    \begin{equation*}
        \left(-1,\frac{4\sqrt{8c^2-6c}}{4c-1}-2\right).
    \end{equation*}
    
    For $k \geq 3$, $\mu=-2$, and we obtain the bound
    \begin{equation*}
        \theta'_{i-1}-\theta'_i \geq \frac{k(ck+c-2)}{ck-1}.
    \end{equation*}
    In this case, $L(G)$ must have no eigenvalues in the interval
    \begin{equation*}
        \left(-1,\frac{k(ck+c-2)}{ck-1}-2\right).
    \end{equation*}
    Finally, using Corollary \ref{cor:LG_spec} we can translate these conditions on the spectrum of $L(G)$ into conditions on the spectrum of $G$. In particular, $G$ must have no eigenvalues in the range
    \begin{align*}
        &\left(-1,-2+\frac{4\sqrt{8c^2-6c}}{4c-1}\right) &\text{ if } k=2, \\
        &\left(-k+1,\frac{ck-k}{ck-1}\right) &\text{ if } k \geq 3. &\qedhere
    \end{align*}
\end{proof}


We can also derive a version of Proposition \ref{prop:hoffman_EN_regular} which holds for general graphs, where now the line graph must have a large gap in its spectrum.

\begin{proposition}\label{prop:hoffman_EN}
    Let $G$ be a graph with maximum degree $\Delta \geq 2$ for which the Hoffman-type bound (Corollary \ref{cor:ratio_chiprime_2}) gives a lower bound larger than $c\Delta^2$ for some $c > 1$. Then $L(G)$ has no eigenvalues in the interval
    \begin{align*}
        &\left(-1,\frac{2\sqrt{c\Delta^3(c\Delta^2-(c+1)\Delta+1}}{c\Delta^2-1}-2\right) &\text{ if } \Delta < 5, \\
        &\left(-1,\frac{c\Delta^3-(c+1)\Delta^2+1}{2c\Delta^2-\Delta-1}\right) &\text{ if } \Delta \geq 5.
    \end{align*}
\end{proposition}

We omit the proof of Proposition \ref{prop:hoffman_EN}, as it is similar to that of Proposition \ref{prop:hoffman_EN_regular}, only much more cumbersome.

A similar analysis can likely be performed for the Hoffman-type bound on $\chi_3'(G)$ (Corollary \ref{cor:ratio_chiprime_3}), albeit more technical.
\subsection*{Acknowledgements}
Aida Abiad is supported by NWO (Dutch Research Council) through the grants VI.Vidi.213.085 and OCENW.KLEIN.475. Harper Reijnders is supported by NWO through the grant VI.Vidi.213.085. The authors thank Hitesh Kumar for a careful reading of the manuscript, as well as for his help with the proof of Proposition \ref{prop:t-crit}.


\newpage

\section*{Appendix}\label{appendix}

\begin{table}[H]
    \centering
    \tiny
    \begin{tabular}{l|cc|c|c|l}
    \hline
    Graph & Corollary \ref{cor:ratio_chiprime_2} / \ref{cor:ratio_chiprime_2_reg} & Theorem \ref{thm:1st_inertial_chiprime} &  Corollary \ref{cor:wilf} & $\chi'_2(G)$ & Polynomial for Theorem \ref{thm:1st_inertial_chiprime} \\
    \hline
        Balaban 10-cage & $5$ & $5$ & $13$ & $6$ & $x^2 + 0.55x$ \\
        \textbf{Balaban 11-cage} & \textbf{6} & $4$ & $13$ & $6$ & $x^2 + 0.53x$ \\
        Bidiakis cube & $6$ & $5$ & $12$ & $8$ & $x^2 + 3x$ \\
        \textbf{Biggs-Smith graph} & \textbf{6} & \textbf{6} & $13$ & $6$ & $x^2 + 0.12x$ \\
        Blanusa First Snark Graph & $6$ & $4$ & $13$ & $7$ & $x^2 + 0.64x$ \\
        Blanusa Second Snark Graph & $6$ & $4$ & $13$ & $7$ & $x^2 + 0.2x$ \\
        Brinkmann graph & $9$ & $5$ & $25$ & $10$ & $x^2 + 0.01x$ \\
        Brouwer-Haemers & $68$ & time & $633$ & time & N/A \\
        \textbf{Bucky Ball} & \textbf{5} & $4$ & $13$ & $5$ & $x^2 + 0.38x$ \\
        Cell 600 & $43$ & time & $160$ & time & N/A \\
        Chvatal graph & $8$ & $4$ & $22$ & $12$ & $x^2$ \\
        \textbf{Clebsch graph} & \textbf{10} & $4$ & $37$ & $10$ & $x^2 -2x$ \\
        Coclique graph of H-S graph & $50$ & time & $365$ & time & N/A \\
        Conway-Smith graph for 3S7 & $29$ & $12$ & $127$ & time & $x^2 -6.99x$ \\
        Coxeter Graph & $6$ & $5$ & $13$ & $7$ & $x^2$ \\
        \textbf{Desargues Graph} & \textbf{5} & $3$ & $13$ & $5$ & $x^2$ \\
        Dejter Graph & $12$ & $6$ & $57$ & time & $x^2 -x$ \\
        \textbf{Dodecahedron} & \textbf{5} & $4$ & $13$ & $5$ & $x^2$ \\
        Double star snark & $5$ & $5$ & $13$ & $7$ & $x^2 + 0.24x$ \\
        \textbf{Durer graph} & \textbf{6} & $4$ & $11$ & $6$ & $x^2 + x$ \\
        \textbf{Dyck graph} & \textbf{6} & $3$ & $13$ & $6$ & $x^2$ \\
        \textbf{Ellingham-Horton 54-graph} & \textbf{6} & $4$ & $13$ & $6$ & $x^2 + 0.49x$ \\
        \textbf{Ellingham-Horton 78-graph} & \textbf{6} & $4$ & $13$ & $6$ & $x^2 + 1.52x$ \\
        Errera graph & $11$ & $9$ & $30$ & $15$ & $x^2 -1.62x$ \\
        F26A Graph & $6$ & $6$ & $13$ & $7$ & $x^2 + 0.17x$ \\
        Flower Snark & $5$ & $4$ & $13$ & $6$ & $x^2 + 0.85x$ \\
        Folkman Graph & $8$ & $8$ & $22$ & $10$ & $x^2$ \\
        Foster Graph & $5$ & $4$ & $13$ & $6$ & $x^2$ \\
        Foster graph for 3.Sym(6) graph & $14$ & $7$ & $51$ & $15$ & $x^2$ \\
        \textbf{Franklin graph} & \textbf{6} & $3$ & $12$ & $6$ & $x^2$ \\
        \textbf{Frucht graph} & \textbf{6} & $5$ & $10$ & $6$ & $x^2 + 1.11x$ \\
        Goldner-Harary graph & $10$ & $7$ & $26$ & $24$ & $x^2 -0.87x$ \\
        Golomb graph & $7$ & $5$ & $15$ & $11$ & $x^2 + 0.85x$ \\
        Gosset Graph & $252$ & $95$ & $643$ & time & $x^2 -22x$ \\
        \textbf{Gray graph} & \textbf{6} & $5$ & $13$ & $6$ & $x^2 + x$ \\
        Grotzsch graph & $7$ & $5$ & $19$ & $10$ & $x^2 -0.34x$ \\
        Harborth Graph & $8$ & $7$ & $18$ & $9$ & $x^2 -0.91x$ \\
        Harries Graph & $6$ & $5$ & $13$ & time & $x^2 + 0.55x$ \\
        \textbf{Harries-Wong graph} & \textbf{6} & $5$ & $13$ & $6$ & $x^2 + 0.55x$ \\
        \textbf{Heawood graph} & \textbf{7} & $4$ & $13$ & $7$ & $x^2$ \\
        Herschel graph & $6$ & $5$ & $14$ & $9$ & $x^2 + 1.88x$ \\
        \textbf{Hexahedron} & \textbf{6} & $3$ & $11$ & $6$ & $x^2 + 9996x$ \\
        Hoffman Graph & $8$ & $7$ & $22$ & $12$ & $x^2 + x$ \\
        Hoffman-Singleton graph & $18$ & $9$ & $85$ & time & $x^2 -4x$ \\
        Holt graph & $8$ & $5$ & $25$ & $9$ & $x^2 + 0.22x$ \\
        \textbf{Horton Graph} & \textbf{6} & $4$ & $13$ & $6$ & $x^2 + 1.53x$ \\
        \textbf{Icosahedron} & \textbf{15} & $8$ & $25$ & $15$ & $x^2$ \\
        Klein 3-regular Graph & $6$ & $4$ & $13$ & $7$ & $x^2 + x$ \\
        \textbf{Klein 7-regular Graph} & \textbf{21} & $10$ & $59$ & $21$ & $x^2 -4x$ \\
        Krackhardt Kite Graph & $7$ & $3$ & $15$ & $14$ & $x^2 + 2.7x$ \\
        \textbf{Ljubljana graph} & \textbf{6} & $5$ & $13$ & $6$ & $x^2$ \\
        M22 Graph & $52$ & time & $436$ & time & N/A \\
        \textbf{Markstroem Graph} & \textbf{6} & \textbf{6} & $10$ & $6$ & $x^2$ \\
        McGee graph & $6$ & $4$ & $13$ & $7$ & $x^2 + 3x$ \\
        Meredith Graph & $7$ & $7$ & $20$ & $13$ & $x^2 -0.67x$ \\
        \textbf{Moebius-Kantor Graph} & \textbf{6} & $3$ & $13$ & $6$ & $x^2$ \\
        Moser spindle & $6$ & $4$ & $10$ & $9$ & $x^2 + 8985.35x$ \\
        Murty Graph & $6$ & $7$ & $11$ & $10$ & $x^2 + 0.38x$ \\
        \textbf{Nauru Graph} & \textbf{6} & $4$ & $13$ & $6$ & $x^2$ \\
        \textbf{Pappus Graph} & \textbf{6} & $4$ & $13$ & $6$ & $x^2$ \\
        Perkel Graph & $15$ & $9$ & $61$ & time & $x^2 -3x$ \\
        \textbf{Petersen graph} & \textbf{5} & $3$ & $13$ & $5$ & $x^2 -3x$ \\
        Poussin Graph & $10$ & $8$ & $28$ & $17$ & $x^2 -1.21x$ \\
        Robertson Graph & $8$ & $4$ & $25$ & $10$ & $x^2 -0.62x$ \\
        Schläfli graph & $108$ & $31$ & $206$ & time & $x^2 -10x$ \\
        \textbf{Shrikhande graph} & \textbf{16} & $7$ & $40$ & $16$ & $x^2 -3x$ \\
        Sims-Gewirtz Graph & $28$ & $14$ & $172$ & time & $x^2 -8x$ \\
        Sousselier Graph & $6$ & $4$ & $18$ & $7$ & $x^2 + 0.37x$ \\
        \textbf{Sylvester Graph} & \textbf{10} & $6$ & $41$ & $10$ & $x^2$ \\
        \textbf{Szekeres Snark Graph} & \textbf{5} & $4$ & $13$ & $5$ & $x^2 + 0.45x$ \\
        Tietze Graph & $6$ & $3$ & $12$ & $7$ & $x^2$ \\
        Tricorn Graph & $5$ & $4$ & $10$ & $7$ & $x^2 + 1.35x$ \\
        \textbf{Truncated Tetrahedron} & \textbf{6} & $5$ & $10$ & $6$ & $x^2$ \\
        \textbf{Tutte 12-Cage} & \textbf{6} & $4$ & $13$ & $6$ & $x^2 + x$ \\
        Tutte-Coxeter graph & $5$ & $5$ & $13$ & $7$ & $x^2 + 3x$ \\
        Twinplex Graph & $6$ & $3$ & $13$ & $7$ & $x^2 + x$ \\
        Wagner Graph & $6$ & $4$ & $11$ & $10$ & $x^2 + 1.41x$ \\
        \textbf{Wells graph} & \textbf{10} & $5$ & $41$ & $10$ & $x^2 + 9.69x$ \\
        Wiener-Araya Graph & $5$ & $5$ & $15$ & $8$ & $x^2 + 0.06x$ \\
    \hline
    \end{tabular}
    \caption{Comparison of $\chi'_2(G)$ bounds for Sage named graphs. Tight bounds are in bold.}
    \label{tab:bound_sim_2}
\end{table}

\begin{table}[H]
    \centering
    \tiny
    \begin{tabular}{l|cc|c|c|l}
    \hline
    Graph & Corollary \ref{cor:ratio_chiprime_3} / \cite[LP (4)]{ANR2024} & Theorem \ref{thm:1st_inertial_chiprime}  & Corollary \ref{cor:wilf} & $\chi'_3(G)$  & Polynomial for Theorem \ref{thm:1st_inertial_chiprime} \\
    \hline
    Balaban 10-cage & $10$ & $7$ & $29$ & time & $x^3 + 1.04x^2 -5.43x$ \\
    Balaban 11-cage & $10$ & $8$ & $29$ & time & $x^3 + 0.63x^2 -5.04x$ \\
    Bidiakis cube & $9$ & $5$ & $17$ & $11$ & $x^3 + x^2 -6x$ \\
    Biggs-Smith graph & $9$ & $6$ & $29$ & time & $x^3 + 0.21x^2 -3.07x$ \\
    \textbf{Blanusa First Snark Graph} & \textbf{14} & $7$ & $22$ & $14$ & $x^3 + 1.06x^2 -6x$ \\
    \textbf{Blanusa Second Snark Graph} & \textbf{14} & $7$ & $21$ & $14$ & $x^3 + 1.38x^2 -4x$ \\
    Bucky Ball & $10$ & $6$ & $25$ & $12$ & $x^3 + 0.04x^2 -5.12x$ \\
    Cell 600 & $120$ & time & $407$ & time & N/A \\
    Coclique graph of H-S graph & $250$ & time & $729$ & time & N/A \\
    Conway-Smith graph for 3S7 & $105$ & $25$ & $301$ & time & $x^3 + 1308.33x^2 -9225.33x$ \\
    \textbf{Coxeter Graph} & \textbf{11} & $5$ & $29$ & $11$ & $x^3 + 0.83x^2 -6x$ \\
    \textbf{Desargues Graph} & \textbf{15} & $6$ & $25$ & $15$ & $x^3 + 0.33x^2 -6x$ \\
    Dejter Graph & $42$ & $16$ & $160$ & time & $x^3 + 4.92x^2 -38.72x$ \\
    \textbf{Dodecahedron} & \textbf{10} & $8$ & $21$ & $10$ & $x^3 + 1.24x^2 -4x$ \\
    Double star snark & $9$ & $5$ & $27$ & $15$ & $x^3 + 1.82x^2 -3.15x$ \\
    Durer graph & $9$ & $9$ & $16$ & $12$ & $x^3 + 0.82x^2 -5.34x$ \\
    \textbf{Dyck graph} & \textbf{12} & $7$ & $27$ & $12$ & $x^3 -x^2 -2x$ \\
    Ellingham-Horton 54-graph & $9$ & $6$ & $26$ & $12$ & $x^3 + 1.27x^2 -5.97x$ \\
    Ellingham-Horton 78-graph & $9$ & $6$ & $26$ & $12$ & $x^3 + 1.02x^2 -5.89x$ \\
    Errera graph & $17$ & $12$ & $44$ & $35$ & $x^3 -4.79x^2 -8.41x$ \\
    \textbf{F26A Graph} & \textbf{13} & $6$ & $27$ & $13$ & $x^3 -1.25x^2 -6x$ \\
    Flower Snark & $10$ & $6$ & $25$ & $15$ & $x^3 -0.28x^2 -5.06x$ \\
    Foster Graph & $9$ & $7$ & $29$ & time & $x^3 + 1.45x^2 -4x$ \\
    \textbf{Foster graph for 3.Sym(6) graph} & \textbf{45} & $11$ & $129$ & $45$ & $x^3 -5.5x^2 -8.5x$ \\
    Frucht graph & $9$ & $6$ & $16$ & $10$ & $x^3 + 0.74x^2 -6x$ \\
    \textbf{Gray graph} & \textbf{9} & $5$ & $29$ & $9$ & $x^3 -2.45x^2 -5.45x$ \\
    Harborth Graph & $12$ & $9$ & $34$ & $17$ & $x^3 -1.17x^2 -8.82x$ \\
    Harries Graph & $11$ & $7$ & $29$ & time & $x^3 + 1.04x^2 -5.43x$ \\
    Harries-Wong graph & $11$ & $7$ & $29$ & time & $x^3 + 1.04x^2 -5.43x$ \\
    Holt graph & $18$ & $8$ & $49$ & time & $x^3 + 2493x^2 -2500x$ \\
    Horton Graph & $9$ & $6$ & $26$ & $12$ & $x^3 + 0.45x^2 -5.52x$ \\
    Klein 3-regular Graph & $10$ & $6$ & $29$ & time & $x^3 + 0.33x^2 -6x$ \\
    \textbf{Krackhardt Kite Graph} & $9$ & $5$ & \textbf{17} & $17$ & $x^3 + 0.47x^2 -1.79x$ \\
    Ljubljana graph & $10$ & $6$ & $29$ & time & $x^3 + 0.29x^2 -5.29x$ \\
    Markstroem Graph & $8$ & $6$ & $18$ & $10$ & $x^3 + 0.57x^2 -6x$ \\
    McGee graph & $9$ & $4$ & $29$ & $13$ & $x^3 -1.89x^2 -6x$ \\
    Meredith Graph & $18$ & $14$ & $34$ & time & $x^3 + 0.5x^2 -7.46x$ \\
    \textbf{Moebius-Kantor Graph} & \textbf{12} & $5$ & $23$ & $12$ & $x^3 -6x$ \\
    \textbf{Nauru Graph} & \textbf{9} & $6$ & $27$ & $9$ & $x^3 + 0.33x^2 -6x$ \\
    \textbf{Pappus Graph} & \textbf{9} & $4$ & $25$ & $9$ & $x^3 -2.33x^2 -6x$ \\
    Perkel Graph & $43$ & $10$ & $161$ & time & $x^3 -5.25x^2 -8.25x$ \\
    Poussin Graph & $16$ & $10$ & $38$ & $37$ & $x^3 -4x^2 -8.98x$ \\
    \textbf{Robertson Graph} & $19$ & $19$ & \textbf{37} & $37$ & $x^3 -1.28x^2 -8.72x$ \\
    Sousselier Graph & $8$ & $6$ & $26$ & $24$ & $x^3 + 4.61x^2 -5.99x$ \\
    \textbf{Sylvester Graph} & \textbf{45} & $9$ & $89$ & $45$ & $x^3 -5x^2 -4x$ \\
    Szekeres Snark Graph & $10$ & $9$ & $23$ & $13$ & $x^3 + 0.1x^2 -4.94x$ \\
    \textbf{Truncated Tetrahedron} & \textbf{9} & $5$ & $17$ & $9$ & $x^3 -4.29x^2 -7x$ \\
    Tutte 12-Cage & $9$ & $9$ & $29$ & time & $x^3 + 1.04x^2 -5.43x$ \\
    Tutte-Coxeter graph & $9$ & $5$ & $29$ & $11$ & $x^3 -x^2 -6x$ \\
    \textbf{Twinplex Graph} & $9$ & $9$ & \textbf{17} & $17$ & $x^3 + 0.15x^2 -4.62x$ \\
    \textbf{Wells graph} & \textbf{40} & $9$ & $79$ & $40$ & $x^3 + 1990x^2 -3992x$ \\
    \end{tabular}
    \caption{Comparison of $\chi'_3(G)$ bounds for Sage named graphs. Tight bounds are in bold.}
    \label{tab:bound_sim_3}
\end{table}

\end{document}